\documentclass[10pt]{amsart}
\usepackage{graphicx}
\pdfoutput=1

\newcommand{\bbz}{\mathbb{Z}}

\newcommand{\bbr}{\mathbb{R}}

\newcommand{\tC}{{\widetilde{C}}}

\newcommand{\RR}{\mathbb{R}}

\newcommand{\AAA}{{\mathcal{A}}}
\newcommand{\CCC}{{\mathcal{C}}}

\newtheorem{theorem}{Theorem}[section]
\newtheorem{corollary}[theorem]{Corollary}
\newtheorem{proposition}[theorem]{Proposition}

\newtheorem{lemma}[theorem]{Lemma}
\newtheorem{defn}[theorem]{Definition}

\newtheorem{example}[theorem]{Example}
\newtheorem{remark}[theorem]{Remark}

\newtheorem{observation}[theorem]{Observation}
\newtheorem{claim}[theorem]{Claim}

\numberwithin{figure}{section}

\begin{document}
\title[Triangulations, Arrangements and Tableaux] {Triangle-Free
Triangulations, Hyperplane Arrangements and Shifted Tableaux}

\author{Ron M. Adin}
\address{Department of Mathematics\\
Bar-Ilan University\\
Ramat-Gan 52900\\
Israel} \email{radin@math.biu.ac.il}

\author{Yuval Roichman}
\address{Department of Mathematics\\
Bar-Ilan University\\
Ramat-Gan 52900\\
Israel} \email{yuvalr@math.biu.ac.il}


\date{August 6, '12}

\begin{abstract}
Flips of diagonals in colored triangle-free triangulations of a
convex polygon are interpreted as
moves between two adjacent chambers in a certain 
graphic hyperplane arrangement. Properties of geodesics in the
associated flip graph  are deduced. In particular, it is shown
that: (1) every diagonal is flipped exactly once in a geodesic
between distinguished pairs of antipodes; (2) the number of
geodesics between these antipodes is equal to twice the number of
standard Young tableaux of a truncated shifted staircase shape.
\end{abstract}

\maketitle

\section{Introduction}


It was shown in~\cite{AFR} that the diameter of the flip graph on
the set of all colored triangle-free triangulations of a convex
$n$-gon (to be defined in Subsection~\ref{2.1}) is exactly
$n(n-3)/2$. Observing that this is the number of diagonals in a
convex $n$-gon, it was conjectured by Richard Stanley that
all diagonals are flipped 
in a geodesic between two
antipodes.

\medskip

In this paper Stanley's conjecture is proved for distinguished
pairs of antipodes (Corollary~\ref{RS} below). The proof applies a
$\widetilde C_n$-action on arc permutations, which yields an embedding
of the flip graph in a graphic hyperplane arrangement. Geodesics
between distinguished antipodes in the flip graph are then
interpreted as minimal galleries from a given chamber $c$ to the
negative chamber $-c$, while diagonals are interpreted as
separating hyperplanes.

\bigskip

The set of geodesics between these antipodes is further studied in
Section~\ref{5}. It is shown that the number of these geodesics is
equal to twice the number of Young tableaux of a truncated shifted
staircase shape. Motivated by this result, product formulas for
this number, as well as for other truncated shapes, were found by
Greta Panova~\cite{Panova} and Ronald C. King and the
authors~\cite{AKR}.

\section{Triangle-Free Triangulations}\label{P3}

In this Section we 
recall basic concepts and main
results from~\cite{AFR}.

\subsection{Basic Concepts}\label{2.1}\ \\

Label the vertices of a convex $n$-gon $P_{n}$ ($n > 4$) by the
elements $0,\ldots,n-1$ of the additive cyclic group $\bbz_{n}$.
Consider a triangulation (with no extra vertices) of the polygon.
Each edge of the polygon is called an {\em external edge} of the
triangulation; all other edges of the triangulation are called
{\em internal edges}, or {\em chords}.

\medskip

\begin{defn}
A triangulation of a convex $n$-gon $P_{n}$ is called
{\em internal-triangle-free}, or simply {\em triangle-free}, 
if it contains no triangle with $3$ internal
edges. The set of all triangle-free triangulations of $P_{n}$ is
denoted $TFT(n)$.
\end{defn}

A chord in $P_{n}$ is called {\em short} if it connects the
vertices labeled $i-1$ and $i+1$, for some $i\in \bbz_{n}$. A
triangulation is triangle-free if and only if it contains only two
short chords~\cite[Claim 2.3]{AFR}.

A {\em proper coloring} (or {\em orientation}) of a triangulation $T\in TFT(n)$ is a
labeling of the chords by $0,\ldots,n-4$ such that
\begin{enumerate}
\item
One of the short chords is labeled $0$.
\item
If a triangle has exactly two internal edges then
their labels are consective integers $i$, $i+1$.
\end{enumerate}
It is easy to see that each $T \in TFT(n)$ has exactly two proper colorings.
The set of all properly colored triangle-free triangulations is denoted $CTFT(n)$.

Each chord in a triangulation is a diagonal of a unique quadrangle
(the union of two adjacent triangles). Replacing this chord by the
other diagonal of that quadrangle is a {\em flip} of the chord. A
flip in a colored triangulation preserves the color of the flipped
diagonal.

\begin{defn}
The {\em colored flip graph} $\Gamma_n$ is defined as follows:
the vertices are all the colored triangle-free triangulations in
$CTFT(n)$. Two triangulations are connected in $\Gamma_n$ by an
edge labeled $i$ if one is obtained from the other by a flip of
the chord labeled $i$.
\end{defn}

%

See Figure~\ref{f.Gamma_7} for a drawing of $\Gamma_7$,
where the coloring of a triangulation is displayed by shading the triangle 
with the short chord labeled $0$ and two external edges as sides.  

\begin{figure}[ht]\label{f.Gamma_7}
\begin{center}
\includegraphics[scale=0.4, trim = 0pt 220pt 0pt 0pt, clip]{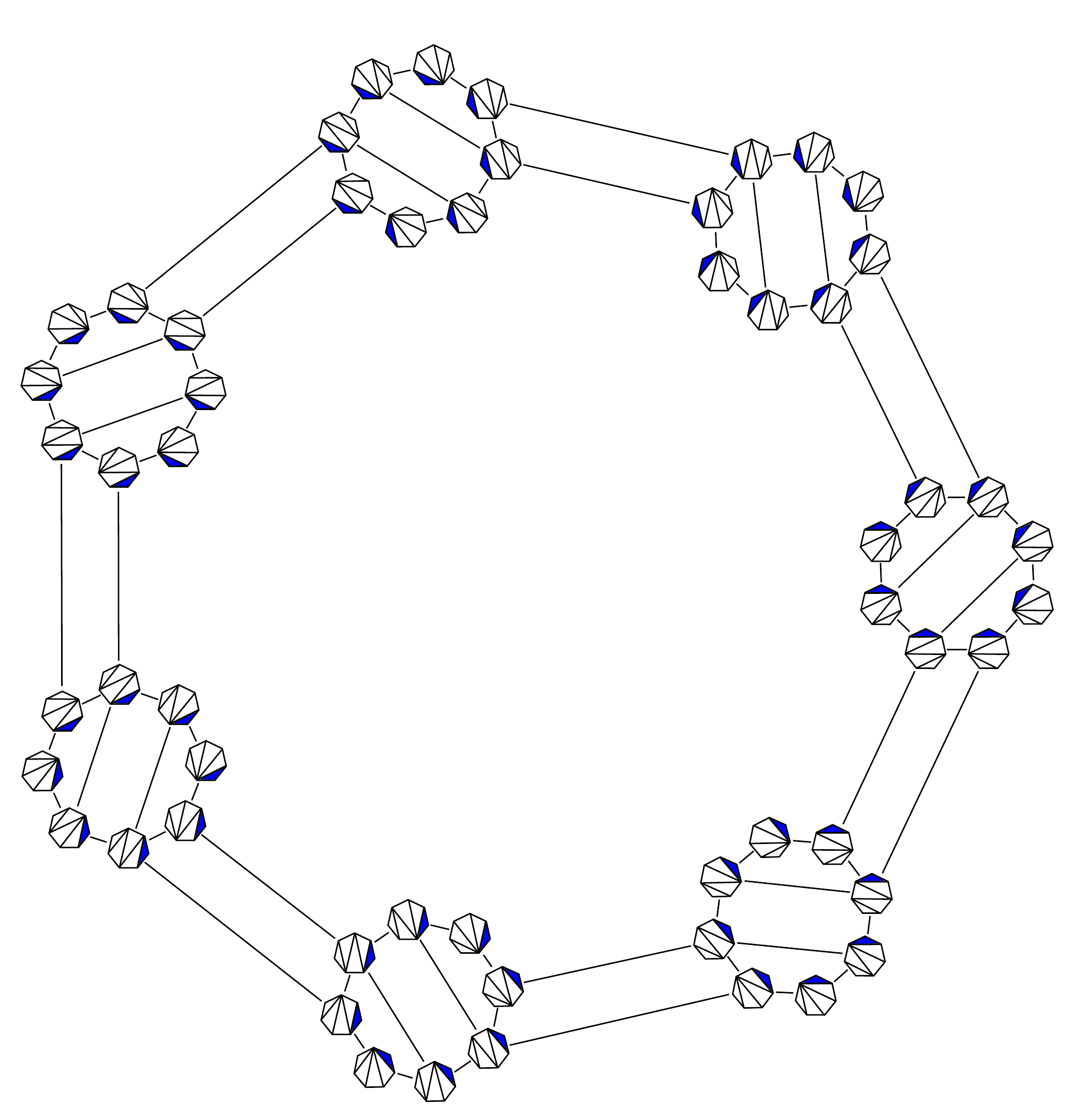}
\caption{$\Gamma_7$} 
\label{fig:Gamma_7}
\end{center}
\end{figure}

\subsection{A $\widetilde C_{n-4}$-Action on Triangle-Free Triangulations}\label{2.2}\ \\

Let $\widetilde C_n$ be the affine Weyl group generated by
$$
S=\{s_0,s_1,\ldots, s_{n-1},s_n\}
$$
subject to the Coxeter relations
\begin{equation}\label{relation1}
s_i^2=1\qquad (\forall i),
\end{equation}
\begin{equation}\label{relation2}
(s_i s_j)^2=1\qquad (|j-i|>1),
\end{equation}
\begin{equation}\label{relation3}
(s_i s_{i+1})^3=1\qquad  (1\le i \le n-2),
\end{equation}
and
\begin{equation}\label{relation4}
(s_i s_{i+1})^4=1 \qquad (i=0,n-1).
\end{equation}

The 
group $\tC_{n-4}$ acts naturally on $CTFT(n)$ by flips: Each generator
$s_i$ flips the chord labeled $i$ in $T\in CTFT(n)$, provided that
the result still belongs to $CTFT(n)$. 
If this is not the case then $T$ is unchanged by $s_i$.

\begin{proposition}\label{t.action1}\cite[Proposition 3.2]{AFR}
This operation determines a transitive $\tC_{n-4}$-action on
$CTFT(n)$.
\end{proposition}

This affine Weyl group action on $CTFT(n)$ was used to calculate
the diameter of $\Gamma_n$.

\begin{theorem}\label{diameter}\cite[Theorem 5.1]{AFR}
The diameter of $\Gamma_n$ $(n>4)$ is $n(n-3)/2$.
\end{theorem}

For any colored triangle-free triangulation $T$, denote by $T^R$
the colored triangle-free triangulation obtained by reversing the
labeling in $T$; namely, the chord labeled $i$ in
$T$ is labeled $n-4-i$ in $T^R$ ($0 \le i \le n-4$).

\begin{theorem}\label{antipodes}\cite[Proposition 5.6]{AFR}
For every $n>4$ and $T\in CTFT(n)$, the distance between $T$ and
$T^R$ in $\Gamma_n$ is exactly $n(n-3)/2$.
\end{theorem}


\section{A $\widetilde C_{n-2}$-Action on Arc Permutations}\label{2nd-action}

\subsection{Arc Permutations}\label{P1}\ \\

Let $S_n$ be the symmetric group on the letters $1,\dots,n$.
Denote a permutation $\pi\in S_n$ by the sequence
$[\pi(1),\dots,\pi(n)]$ and transpositions by $(i,j)$.

{\em Intervals} in the cyclic group $\bbz_n$  are subsets of the form
$\{i, i+1, \dots, i+k\}$, where addition is modulo $n$.


\begin{defn} A permutation $\pi\in S_n$ is an {\em arc permutation}
if, for every $1\le k\le n$, the first $k$ letters in $\pi$
form an interval in $\bbz_n$ (where $n\equiv 0$, namely, the
letter $n$ is identified with zero).
\end{defn}

\noindent{\bf Example.} $\pi = [1,2,5,4,3]$ is an arc permutation
in $S_5$, but $\pi = [1,2,5,4,3,6]$ is not an arc permutation in
$S_6$, since $\{1,2,5\}$ is an interval in $\bbz_5$ but not in
$\bbz_6$.

\medskip

The following claim is obvious.

\begin{claim}
The number of arc permutations in $S_n$ is $n2^{n-2}$.
\end{claim}

\begin{proof}
There are $n$ options for $\pi(1)$ and two options for every other
letter except the last one.
\end{proof}

Denote by $U_n$ the set of arc permutations in $S_n$.

\begin{defn}
Define $\phi:U_n \to \bbz_n\times \bbz_2^{n-2}$ as follows:
\begin{enumerate}
\item
$$
\phi(\pi)_1:=\pi(1).
$$
\item
For every $2 \le i\le n-1$, if $\{\pi(1),\dots,\pi(i-1)\}$ is the arc $[k,m]$
then $\pi(i)$ is either $k-1$ or $m+1$. Let
$$
\phi(\pi)_i := \begin{cases}
0, &\hbox{\rm if } \pi(i)=k-1;\\
1, &\hbox{\rm if } \pi(i)=m+1.
\end{cases}
$$
\end{enumerate}
\end{defn}
$\phi$ is clearly a bijection.

\subsection{A $\widetilde C_{n-2}$-Action}\label{Cn-action}\ \\

Let $\{\sigma_i:\ 1\le i\le n-1\}$ be the Coxeter generating set of the
symmetric group $S_n$, where $\sigma_i$ is identified with the
adjacent transposition $(i,i+1)$.

\begin{defn}
For every $0\le i\le n-2$ define a map $\rho_i:U_n\to U_n $ as
follows:
$$
\rho_i(\pi)=\begin{cases}
\pi \sigma_{i+1},&\hbox{\rm if }\pi \sigma_{i+1}\in U_n;\\
\pi, &\hbox{\rm otherwise.}
\end{cases} \qquad(\forall \pi\in U_n)
$$
\end{defn}

Note that, for $\pi\in U_n$, $\pi \sigma_{i+1}\in U_n$ iff
either $i\in\{0,n-2\}$ or $\phi(\pi)_{i+1} \ne \phi(\pi)_{i+2}$.


\begin{observation}\label{action-on-vectors}
For every  $\pi\in U_n$ and $1 \le j \le n-1$,
\[
\phi(\rho_0(\pi))_j = \begin{cases}
\phi(\pi)_1 - 1 \pmod n, &\hbox{\rm if } j=1 \ {\rm{and}}\ \phi(\pi)_2=0;\\
\phi(\pi)_1 + 1 \pmod n, &\hbox{\rm if } j=1 \ {\rm{and}}\ \phi(\pi)_2=1;\\
\phi(\pi)_2 + 1 \pmod 2, &\hbox{\rm if } j=2;\\
\phi(\pi)_j, &\hbox{\rm if } j\ne 1,2,\\
\end{cases}
\]
\[
\phi(\rho_i(\pi))_j = \phi(\pi)_{\sigma_{i+1}(j)}\qquad (1 \le i \le n-1,\ \forall j)
\]
and
\[
\phi(\rho_{n-2}(\pi))_j = \begin{cases}
\phi(\pi)_j, &\hbox{\rm if } j\ne n-1;\\
\phi(\pi)_{n-1}+1 \pmod 2, &\hbox{\rm if } j= n-1.
\end{cases}
\]
\end{observation}

\begin{proposition}\label{t.action}
The maps $\rho_i$, when extended multiplicatively, determine a well defined transitive 
$\widetilde C_{n-2}$-action on the set $U_n$ of arc permutations.
\end{proposition}

\begin{proof}
To prove that the operation is a $\tC_{n-2}$-action,
it suffices to show that it is consistent with the 
Coxeter relations defining $\widetilde{{C}}_{n-2}$ when the operator
$\rho_i$ is interpreted as an action of the generator $s_i$.
All relations may be easily verified using Observation~\ref{action-on-vectors};
we leave the details to the reader.

\smallskip

To prove that the action is transitive, notice first that 
$\rho_0(\pi)(1)=\pi(2)= {\pi(1)\pm 1 \pmod n}$. It thus suffices to prove that,
for every $1\le k\le n$, the maximal parabolic subgroup
$\langle s_1,\ldots,s_{n-2}\rangle$ of $\tC_{n-2}$ acts
transitively on the set $U_n^{(k)}:=\{\pi\in U_n:\ \pi(1)=k\}$.
Indeed, this parabolic subgroup is isomorphic to the classical
Weyl group $B_{n-2}$. By Observation~\ref{action-on-vectors}, the
restricted $B_{n-2}$-action on  $U_n^{(k)}$ may be identified with
the natural $B_{n-2}$-action on all subsets of 
$\{1,\ldots,n-2\}$, and is thus transitive.

\end{proof}

\section{A Graphic Hyperplane Arrangement}

\subsection{Real Hyperplane Arrangements}\label{P2}\ \\

Let $\AAA$ be an arrangement of finitely many linear hyperplanes
in $\RR^d$ that is {\it central} and {\it essential}, meaning that
$\cap_{H \in \AAA} H = \{0\}$.  Let $L=\sqcup_{i=0}^d L_i$ be
the corresponding graded poset of intersection subspaces,
ordered by reverse inclusion. $L$ is a geometric lattice.

Let $\CCC$ be the set of {\em chambers} of $\AAA$, namely 
the connected components of the complement $\RR^d \setminus \cup_{H \in \AAA} H$.
Define a graph structure $G_1(\AAA)$ on the set of vertices $\CCC$,
with two chambers $c, c' \in \CCC$ connected by an edge
if they are separated by exactly one hyperplane in $\AAA$.
It is well-known that the diameter of this graph is equal to 
the number of hyperplanes, $|L_1|=|\AAA|$.

\medskip

The {\em reflection arrangement} $\AAA_{n-1}$ of type $A_{n-1}$,
corresponding to the symmetric group $S_n$, has as ambient space 
the $(n-1)$-dimensional subspace
\[
W = \{\bar x = (x_1,\ldots,x_n) \in \RR^n\,|\,x_1 + \ldots + x_n = 0\}
\]
of $\RR^n$.
Its hyperplanes are $H_{ij}:=\{\bar x \in W\,|\,x_i = x_j\}$ for $1 \leq i < j \leq n$.
The chambers may be identified with permutations in $S_n$, via
\[
c_\pi:=\{\bar x\in W\,|\, x_{\pi(1)}<x_{\pi(2)}<\cdots<x_{\pi(n)}\}\qquad(\forall \pi \in S_n).
\]
The symmetric group $S_n$ acts on the chambers
via $\sigma_i(c_\pi):= c_{\pi \sigma_i}$, the unique chamber which
is separated from $c_\pi$ only by the hyperplane $H_{\pi(i),\pi(i+1)}$.
Then, for every $\pi\in S_n$, $c_\pi$ and $-c_\pi=c_{\pi w_0}$ are
antipodes in the graph $G_1(\AAA_{n-1})$, where $w_0 :=
[n,n-1,\ldots,1]$ is the longest element in $S_n$.

\medskip

A (simple undirected) graph $G=(V,E)$ of order $n$ consists of a
set $V=\{v_1,\dots,v_n\}$ of vertices and a set $E$ of edges,
which are unordered pairs of distinct vertices. 
The associated {\em graphic arrangement} $\AAA(G)$ is the hyperplane arrangement
in $W \cong \RR^{n-1}$ defined by
$$
\AAA(G):=\{H_{ij}\,|\,\{v_i,v_j\}\in E\} \subseteq \AAA_{n-1}.
$$
For example, if $K_n$ be the complete graph of order $n$ 
then the associated graphic arrangement  $\AAA(K_n)$ is 
the whole reflection arrangement $\AAA_{n-1}$.  For more information
see~\cite{OT}.



\subsection{The Graph of Chambers $G_1(U'_n)$}\label{H1}\ \\

\begin{defn}
Let $G$ be a graph of order $n$.
Two permutations $\pi,\tau\in S_n$ are {\em $G$-equivalent} if
the points $(\pi(1),\dots,\pi(n)), (\tau(1),\dots,\tau(n))\in W$ 
lie in the same chamber of the associated graphic arrangement $\AAA(G)$.
\end{defn}


Index by $1,\ldots,n$ the vertices of the complete graph $K_n$,
and consider the graph $K'_n$ obtained by deleting the edges
$\{1,2\}$, $\{2,3\}$, $\ldots$, $\{n-1,n\}$ and $\{n,1\}$
from $K_n$. Let $\AAA'_{n-1}:=\AAA(K'_n)$ be
the associated graphic arrangement. Two permutations $\pi,\tau\in
S_n$ are $K'_n$-equivalent if and only if there exist permutations
$\pi=\pi_0, \pi_1,\dots, \pi_t=\tau$ such that, 
for every $0\le r\le t-1$, there exists $1\le i \le n-1$ such that 
$\pi_{r+1}=\pi_r \sigma_i$ and 
$\pi_r \sigma_i \pi_r^{-1}\in \{\sigma_j:\ 1\le j \le n-1\}\cup\{(1,n)\}$.
In other words, the $K'_n$-equivalence is the transitive closure
of the the following relation: there exist $1\le j<n$ such that
the letters $j$ and $j+1$, or $1$ and $n$, are
 adjacent in $\pi$, and $\tau$ is obtained from $\pi$ by switching their positions.

\begin{remark}\ \rm
Since $K'_n$ is invariant under the natural action of the dihedral
group $I_2(n)$, this group may be embedded in the automorphism
group of the graph $G_1(\AAA'_{n-1})$. Indeed, let $\gamma$ be the
cycle $(1,2,\dots,n)\in S_n$ and $w_0:=[n,n-1,\dots,1]$ the
longest element in $S_n$.
If $A$ is a $K'_n$-equivalence class then for $0\le j<n$ and
$\epsilon\in \{0,1\}$, $w_0 \gamma^j A$ is a also
$K'_n$-equivalence class. Moreover, edges in $G_1(\AAA'_{n-1})$
are indexed by pairs of $K'_n$-equivalence classes, where for
every such a pair, $(A,B)$ is an edge in $G_1(\AAA'_{n-1})$ if and
only if $w_0^\epsilon \gamma^j (A,B)$ is an edge in
$G_1(\AAA'_{n-1})$.
\end{remark}

\medskip


\begin{defn}
\begin{itemize}
\item[(i)] Define {\em $\tilde K'_n$-equivalence} on the subset of
arc permutations $U_n\subset S_n$ as the transitive closure of the
relation: there exist $1\le j<n$ such that the letters $j$ and
$j+1$, or $1$ and $n$, are
 adjacent in $\pi$, and $\tau$ is obtained from $\pi$ by switching their positions.

\item[(ii)]
 Let $U'_n$ be the set of $\tilde K'_n$-equivalence classes in $U_n$.

\item[(iii)] Let $G_1(U'_n)$ be the graph whose vertex set is
$U'_n$; two $\tilde K'_n$-equivalence classes in $U'_n$ are
adjacent in $G_1(U'_n)$ if they have representatives, whose
corresponding chambers in $G_1(\AAA_{n-1})$ lie in adjacent
chambers in $G_1(\AAA'_{n-1})$.
\end{itemize}
\end{defn}

\begin{observation}\label{classes}
For $n> 3$ all $\tilde K'_n$-equivalence classes in $U_n$ consist
of four permutations $\{\pi, \pi \sigma_1, \pi \sigma_{n-1}, \pi
\sigma_1 \sigma_{n-1}\}$.
\end{observation}

\begin{proof}
For every $\pi\in U_n$ and $1<i<n-2$, if $\pi(i+1)=\pi(i)\pm 1$
then $\pi \sigma_i\not\in U_n$. On the other hand, for every
$\pi\in U_n$ and $i\in\{1,n-1\}$,  $\pi \sigma_i\in U_n$.
\end{proof}

\medskip

Note that, by definition, two  $\tilde K'_n$-equivalent arc
permutations are $K'_n$-equivalent in $S_n$; hence, they lie in
same chamber in $G_1(\AAA'_{n-1})$. One concludes that $G_1(U'_n)$
contains no loops.


\medskip

\begin{example}\ \rm
\begin{itemize}
\item[(a)] For $n=4$ there are four $K'_4$-equivalence classes in
$S_4$:

\smallskip

\noindent ${\bf 1234}=\{[1234], [1324], [2134], [1243], [2143],
[2413]\}$ and its images under cyclic rotations $\gamma^j {\bf
1234}$, $0\le j<4$.

\smallskip

\noindent The edges in the graph $G_1(\AAA'_3)$ are all cyclic
rotations of $({\bf 1234},{\bf 2341})$, thus the graph is a
$4$-cycle. Since, each $K'_4$-equivalence class contains one
$\tilde K'_4$-class in $U'_4$, the graphs $G_1(\AAA'_3)$ and
$G_1(U'_4)$ are identical.

\medskip

\item[(b)] For $n=5$ there are three types of $K'_4$-equivalence
classes in $S_5$:

\smallskip

\noindent ${\bf 12345}=\{[12345], [13245], [21345], [12435],
[21435], [24135], [13254], [21354]\}$ and  its ten images under
dihedral group action $w_0\gamma^j {\bf 12345}$, $0\le j<5$,
$\epsilon\in \{0,1\}$;

\noindent ${\bf 13452}=\{[13452], [14352], [13542]\}$ and  its ten
images under dihedral group action;

\noindent ${\bf 13524}=\{[13524]\}$ and  its ten images  under
dihedral group action.

\smallskip

\noindent The edges in the graph $G_1(\AAA'_4)$ are \\ $\{({\bf
12345},{\bf 32154}), ({\bf 12345}, {\bf 13524}), ({\bf 12345},{\bf
24135}), ({\bf 12345},{\bf 14325}),$
\\ $ ({\bf 13524},{\bf 13452}), ({\bf
13524},{\bf 35124}), ({\bf 13452},{\bf 14325})\}$ and their images
under the dihedral group action.

There are ten $\tilde K'_4$-equivalence classes in $U_4$, each
contained in one of the images under the dihedral group action of
${\bf 12345}$.
Thus the graph $G_1(U'_5)$ is a $10$-cycle.

\end{itemize}

\end{example}

\section{Stanley's Conjecture}

It was conjectured by Richard Stanley~\cite{St1} that
all diagonals are flipped 
in a geodesic between two antipodes in the flip graph $\Gamma_n$
of colored triangle-free triangulations. A bijection between the
set of triangle-free triangulations in $CTFT(n)$ and the subset
$U'_{n}$ of chambers in the graphic hyperplane
arrangement $\AAA(K'_{n})$, 
which preserves the underlying graph structure, is applied to
prove Stanley's conjecture.


\begin{theorem}\label{main111}
The 
flip graph $\Gamma_n$ (without edge labeling) is isomorphic to the
 graph of chambers $G_1(U'_{n})$.
\end{theorem}

Furthermore,

\begin{theorem}\label{main112}
There exists an edge-orientation of the flip graph $\Gamma_n$
such that, for any oriented edge of adjacent triangulations $(T,
S)$, $S$ is obtained from $T$ by flipping the diagonal $[i,j]$
if and only if the corresponding chambers are separated by the
hyperplane $x_i=x_{j}$.
\end{theorem}

An affirmative answer to Stanley's conjecture follows.

\begin{corollary}\label{RS}
For every colored triangle-free triangulation  $T\in CTFT(n)$,
every diagonal is flipped exactly once along the shortest path
from $T$ to same triangulation with reversed coloring $T^R$.
\end{corollary}



\section{Proof of Theorem~\ref{main111}}

\subsection{A $\widetilde C_{n-4}$-Action on $U'_n$}\label{action-on-chambers}\ \\

For every $\pi\in U_n$ denote the $K'_n$-class of $\pi$ in $U_n$
by $\bar\pi$. By Observation~\ref{classes}, for every $\pi\in
U_n$, $\bar\pi\in U'_n$ may be represented by a series of $n-2$
subsets: $\{\pi(1),\pi(2)\},
\{\pi(3)\},\dots,\{\pi(n-2)\},\{\pi(n-1),\pi(n)\}$, where all
subsets except of the first and the last are singletons.

For every simple reflection $\sigma_i\in S_{n-2}$, $1<i<n-3$, let
 $\bar\pi \sigma_i$ be the series of subsets obtained
from $\bar\pi$ by replacing the letters in the $i$-th and $i+1$-st
subsets. Let $\bar\pi \sigma_1$ be obtained from $\bar\pi$ by
replacing letters in the first two subsets as follows: if
$\{\pi(1),\pi(2)\}=\{\pi(3)-2, \pi(3)-1\}$ then the first two
subsets in $\bar\pi \sigma_1$ are $\{\pi(3),\pi(3)-1\},
\{\pi(3)-2\}$; $\{\pi(1),\pi(2)\}=\{\pi(3)+1, \pi(3)+2\}$ then the
first two subsets in $\bar\pi \sigma_1$ are
 $\{\pi(3),\pi(3)+1\}, \{\pi(3)+2\}$. Similarly, $\bar\pi \sigma_{n-3}$ is
obtained from $\bar\pi$ by replacing the letter in the $n-3$-rd
subset with $\pi(n-2)-2$ if $\{\pi(n-1),\pi(n)\}=\{\pi(n-2)-2,
\pi(n-2)-1\}$ and with $\pi(n+2)$ otherwise.

For every $0\le i\le n$ let $\theta_i:U'_n\mapsto U'_n $ be
$$
\theta_i(\bar\pi)=\begin{cases} \bar\pi \sigma_{i+1},&\ {\rm{if}}\ \bar\pi \sigma_{i+1}\in U'_n,\\
\bar\pi, &\ {\rm{if}}\  \bar\pi \sigma_{i+1}\not\in U'_n. \\
\end{cases} \qquad(\forall \bar\pi\in U'_n)
$$

\begin{observation}
The maps $\theta_i$, $(0\le i\le n-4)$, when extended
multiplicatively, determine a well defined transitive $\widetilde C_{n-4}$-action on $U'_n$.
\end{observation}

Proof is similar to the proof of Observation~\ref{t.action} and is
omitted.

\begin{observation}\label{adjacent-chambers}
Two chambers in $\bar\pi,\bar\tau\in U'_n$ are adjacent in
$G_1(U'_n)$ if and only if there exist $0\le i\le n-4$, such that
$\theta_i(\bar\pi)=\bar\tau$.
\end{observation}

\subsection{A Graph Isomorphism}\label{iso}\ \\

Let $f:CTFT(n)\mapsto U'_n$ be defined as follows: if $[a,a+2]$ is
the short chord labeled $0$ then let
$\{\pi(1),\pi(2)\}=\{a,a+1\}$. For $0<i<n-4$,
 assume that the chord
labeled $i-1$ in $T$ is $[a-k,a+m]$ for some $k,m\ge 1$, $k+m =
i+1$. The chord labeled $i$ is then either $[a-k-1,a+m]$ or
$[a-k,a+m+1]$. Let $i+1$-st subset be $\{a-k-1\}$ in the former
case and $\{a+m\}$ in the latter. Finally, let the last subset
consist of the remaining two letters.

\begin{claim}
The map $f: CTFT(n) \mapsto U'_n$ is a bijection.
\end{claim}

\begin{proof}
The map $f$ is invertible.
\end{proof}

Recall the definition of $T^R$ from Section~\ref{P3}.

\begin{observation}\label{reverse}
For every $T\in CTFT(n)$, $f(T^R)$ is obtained from $f(T)$ by
reversing the order of the subsets.
\end{observation}

\medskip

Recall from
Subsection~\ref{2.2} the affine Weyl group $\tC_{n-4}$-action on
$CTFT(n)$.


\medskip

To complete the proof of Theorem~\ref{main111} it suffices to show
that

\begin{proposition}
For every Coxeter generator $s_i$ of $\widetilde C_{n-4}$ $(0\le i\le
n-4)$ and $T\in CTFT(n)$
$$
f(s_i T) = s_i f(T),
$$
where $s_i f(T):=\theta_i(f(T))$.
\end{proposition}

\begin{proof}
For $i=0$, let $[a,a+2]$ be the short chord labeled $0$ in $T$.
Then the chord labeled $1$ is either $[a,a+3]$ or $[a-1,a+2]$. In
the first case the short chord labeled $0$ in $s_0 T$ is
$[a+1,a+3]$ and all other chords are unchanged, in particular, the
chord labeled $1$ in $s_0 T$ is $[a,a+3]$. By definition of the
map $f$, the first two subsets in $f(T)$ are $\{a,a+1\},\{a+2\}$
and the first two subsets in $f (s_0 T)$ are $\{a+1,a+2\}, \{a\}$
and the rest are not changed. On the other hand, by definition of
$\theta_0$, the first two subsets in $s_0 f(T)$ are
$\{a+1,a+2\},\{a\}$ and the rest are unchanged. A similar analysis
shows that  $f(s_0 T) = s_0 f(T)$  when the chord labeled $1$ is
$[a-1,a+2]$.

For $0<i<n-4$ let the chord labeled $i-1$ in $T$ be $[a-k,a+m]$
for some $k,m\ge 1$, $k+m = i+1$. The chords labeled $i$ and $i+1$
are then either $[a-k-1,a+m],[a-k-2,a+m]$ respectively, or
$[a-k,a+m+1],[a-k,a+m+2]$ or
 $[a-k-1,a+m], [a-k-1,a+m+1]$  or
$[a-k,a+m+1], [a-k-1,a+m+1]$. In the first two cases $s_i T=T$, so
$f(s_i T)=f(T)$. On the other hand, in these cases
$f(T)\sigma_i\not \in U'_n$, so $s_i f(T)=f(T)$.

If the chords labeled $i$ and $i+1$ in $T$ are
 $[a-k-1,a+m], [a-k-1,a+m+1]$ respectively, then the chords labeled $i$ and $i+1$ in $s_i T$
are $[a-k,a+m+1], [a-k-1,a+m+1]$. So, the $i$-th and $i+1$-st
subsets in $f(T)$ are $\{a-k-1\},\{a+m\}$, and they are switched
in $s_i f(T)$, so same as the corresponding subsets in $f(s_i T)$.
The proof of the forth case is similar.

Finally, by Observation~\ref{reverse}, $f(s_0 T) = s_0 f(T)$
implies that $f(s_{n-4} T) = s_{n-4} f(T)$.

\end{proof}

\section{Proof of Theorem~\ref{main112}}

\subsection{Orienting the Colored Flip Graph}\label{LFG}\ \\

The goal of this subsection is to equip the colored flip graph
$\Gamma_n$ with an edge orientation that will be used to encode
the location of the flipped diagonals. It will be proved later
that this orientation satisfies the conditions of
Theorem~\ref{main112}. Our starting point is the
edge labeling, 
mentioned in Section~\ref{P3}, which encodes the order of the
chords.

\medskip

Recall from~\cite{AFR} the bijection
$$
\varphi:CTFT(n) \to \bbz_{n}\times \bbz_2^{n-4}
$$
defined as follows: Let $T\in CTFT(n)$. If the (short) chord
labeled $0$ in $T$ is $[a-1,a+1]$ for $a\in \bbz_{n}$, let
$\varphi(T)_0:= a$. For $1\le i\le n-4$, assume that the chord
labeled $i-1$ in $T$ is $[a-k,a+m]$ for some $k,m\ge 1$, $k+m =
i+1$. The chord labeled $i$ is then either $[a-k-1,a+m]$ or
$[a-k,a+m+1]$. Let $\varphi(T)_{i}$ be $0$ in the former case and
$1$ in the latter.


By definition of the map $\varphi$,

\begin{claim}\label{obs-varphi-0}
For every vector $v=(v_0,\dots,v_{n-4}) \in \bbz_{n}\times \bbz_2^{n-4}$ 
and every $0\le i\le n-4$, the diagonal labeled $i$
in the triangulation $T=\varphi^{-1}(v)$ is $[k,m]$ where 
\[
k := v_0-1 - i + \sum\limits_{j=1}^i v_i \in \bbz_n
\]
and
\[
m := v_0+1+\sum\limits_{j=1}^i v_i \in \bbz_n.
\]
Here $0, 1 \in \bbz_2$ are interpreted as $0, 1 \in \bbz_n$.
\end{claim}


It follows that

\begin{corollary}\label{varphi-reverse}~\cite[Lemma 5.7]{AFR}
For every $T\in CTFT(n)$, if $\varphi(T)=(v_0,\dots,v_n)$ then
$\varphi(T^R)_0 = 2+\sum\limits_{i=0}^{n-4}v_i  \in \bbz_n$ and
$\varphi(T^R)_i=1-v_{n-3-i} \in \bbz_2$ $(1\le i \le n-4)$.
\end{corollary}


\begin{observation}\label{obs-varphi-1}~\cite[Observation 3.1]{AFR}
For every  $T\in CTFT(n)$ and a Coxeter generator $s_i$ of $\widetilde C_{n-4}$
$$
(\varphi(s_0T))_j=\begin{cases} \varphi(T)_j,&\ {\rm{if}}\ j\ne 0,1,\\
\varphi(T)_0+1 \pmod n, &\ {\rm{if}}\  j= 0 \ {\rm{and}}\ \varphi(T)_1=0, \\
\varphi(T)_0-1 \pmod n, &\ {\rm{if}}\  j= 0 \ {\rm{and}}\ \varphi(T)_1=1, \\
\varphi(T)_1+1 \pmod 2, &\ {\rm{if}}\  j= 1 \ {\rm{and}}\ \varphi(T)_1=0; \\
\end{cases}
$$
$$
(\varphi(s_{n-4}T))_j=\begin{cases} \varphi(T)_j,&\ {\rm{if}}\  j\ne n,\\
\varphi(T)_n+1 \pmod 2, &\ {\rm{if}}\  j= n;
\end{cases}
$$
and
$$
(\varphi(s_i T))_j=\varphi(T)_{\sigma_i(j)}\qquad (0<i<n-4);
$$
where $\sigma_i:=(i,i+1)$ the adjacent transposition.

\end{observation}

We use this observation to orient the edges in $\Gamma_n$.

\begin{defn}\label{orient}
Orient the edges in $\Gamma_n$ as follows: If the diagonal labeled
$n-4$ is flipped orient the corresponding edge from the
triangulation encoded by last entry $0$ to the one with last entry
$1$. If the flip is of the diagonal labeled $0<i<n-4$ orient the
edge from $T$ with $\varphi(T)_i=0, \varphi(T)_{i+1}=1$ to the one
with these two entries switched; if it flips the diagonal labeled
$0$ orient it by the first entry from $T$ with $\varphi(T)_0=j$ to
the one with first entry under $\varphi$ being $j+1$; .
\end{defn}

See Figure~\ref{f.CTFT6_oriented} for the orientation of $\Gamma_6$,
where each colored triangulation $T$ is labeled by the vector $\varphi(T)$.  

\begin{figure}[ht]\label{f.CTFT6_oriented}
\begin{center}
\includegraphics[scale=0.8, trim = 75pt 325pt 0pt 180pt, clip]{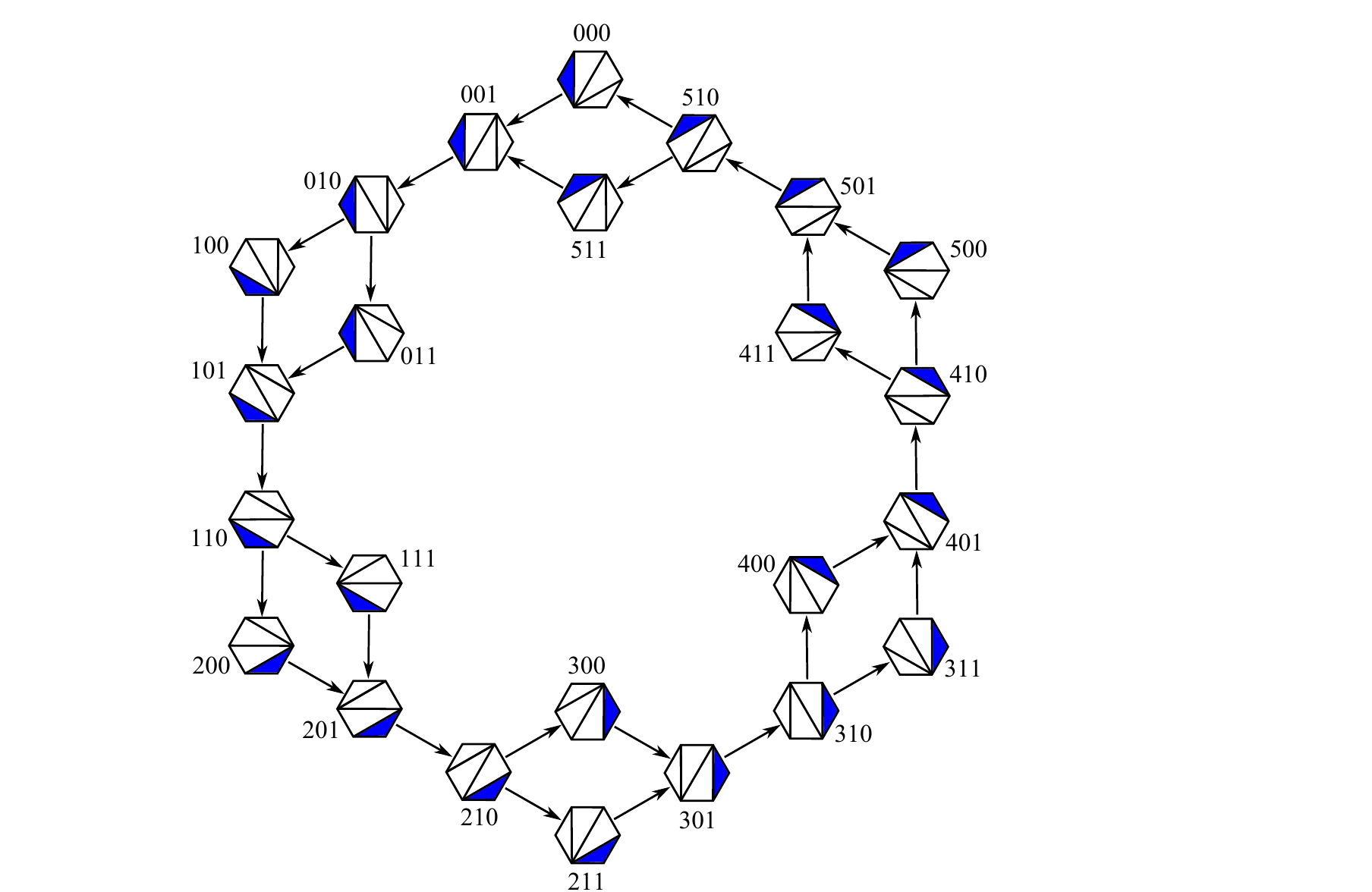}
\caption{$\Gamma_6$ with orientation} 
\label{fig:Gamma_6_oriented}
\end{center}
\end{figure}

\begin{lemma}\label{coherent-orientations}
For every $T\in CTFT(n)$, the orientation of the edges along any
geodesic from $T$ to $T^R$ is coherent with the orientation of
$\Gamma_n$ described in Definition~\ref{orient};
namely, all edges in
a geodesic have the same orientation as in the oriented $\Gamma_n$
or all have the opposite orientation.
\end{lemma}

\begin{proof}
Consider the dominance order on vectors in $\bbz_{n}\times
\bbz_2^{n-4}$; namely,
\[
(v_0,\dots,v_{n-4})\le (u_0,\dots,u_{n-4})
\]
if and only if 
\[
\sum\limits_{i=0}^k v_i\le \sum\limits_{i=0}^k u_i \qquad(0 \le k \le n-4),
\]
where $0, \ldots, n-1 \in \bbz_n$ are interpreted as $0, \ldots, n-1 \in \bbz$,
and similarly for $\bbz_2$.

The resulting poset is ranked by
$$
\ell(v_0,\dots,v_{n-4}):=\sum\limits_{i=0}^{n-4} (n-3-i)v_i.
$$
Using Corollary~\ref{varphi-reverse}, the reader can verify that
for every $T\in CTFT(n)$ 
$$
\ell(\varphi(T^R))-\ell(\varphi(T)) \equiv {n(n-3)}/{2} \pmod {n(n-3)},
$$
which is the distance between $T$ and $T^R$ (Theorem~\ref{antipodes}).

Finally, notice that for every edge $e=(S_1, S_2)$ in $\Gamma_n$,
the edge $e$
is oriented from $S_1$ to $S_2$ if and only if
$$
\ell(\varphi(S_2))-\ell(\varphi(S_1)) \equiv 1 \pmod {n(n-3)}
$$


One concludes that either all steps in a geodesic increase the
rank function by one modulo $n(n-3)$ or all steps decrease it by one. Hence the
lemma holds.

\end{proof}

\medskip

We note that this proof essentially appears (implicitly)
in~\cite{AFR}, where an algebraic interpretation of the rank
function as a length function on $\widetilde C_{n-4}$ is given; 
see, in particular, ~\cite[Sections 3.3 and 5.2]{AFR}.

\bigskip

Now color each edge $(S_1,S_2)$ of $\Gamma_n$, oriented from $S_1$ to $S_2$,
by the chord $[i,j]$ which is erased from $S_1$.
Ignore the edge-orientation and let $\hat\Gamma_n$ be the
resulting edge-labeled flip graph.

\subsection{Edge-Colored Graph Isomorphism}\ \\

Consider an edge-labeled version of the graph $G_1(U'_n)$, denoted
by $\hat G_1(U'_n)$,  where the edge between two adjacent chambers
is labeled by the separating hyperplane.

\begin{theorem}\label{main}
The edge-labeled graphs $\hat G_1(U'_{n})$ and $\hat\Gamma_n$ are
isomorphic.
\end{theorem}

Note that this theorem implies Theorem 
~\ref{main112}.


\medskip

\begin{proof}

By Observation~\ref{adjacent-chambers}, two chambers
$\bar\pi,\bar\tau\in U'_n$ are adjacent in $G_1(U'_n)$ if and only
if there exist corresponding arc permutations $\pi,\tau\in U_n$
and  $1< i< n-1$, such that $\pi \sigma_i=\tau$.
The separating hyperplane is then $x_k=x_m$ if and only if $(k,m)
\pi=\tau$, or equivalently $\pi \sigma_i \pi^{-1}=(k,m)$, for the
transposition $(k,m)\in S_n$.

\medskip

Recall the bijection $f:CTFT(n)\mapsto U'_n$, defined in
Subsection~\ref{iso}. Since $f$ induces a graph isomorphism,
for every $1<i<n-1$, if $\pi, \pi \sigma_i$ are two arc
permutations then $f^{-1}(\bar\pi), f^{-1}(\overline{\pi
\sigma_i}))$ forms an edge in $\Gamma_n$. In order to prove
Theorem~\ref{main112}, it suffices to show that $f^{-1}(\bar\pi
\sigma_i))$ is obtained from $f^{-1}(\bar\pi)$ by flipping the
diagonal $[k,m]$, when the edge is oriented from $f^{-1}(\bar\pi)$
to $f^{-1}(\bar\pi \sigma_i)$.

\medskip

Indeed, an edge is oriented from $f^{-1}(\bar\pi)$ to
$f^{-1}(\overline{\pi \sigma_1})$ if and only if the latter
triangulation is obtained from the first by flipping the short
chord labeled $0$ $[a-1,a+1]$; namely, by replacing the diagonal
$[a-1,a+1]$ by $[a,a+2]$, where the diagonal labeled $1$ is
$[a-1,a+2]$. By definition of the map $f$ 
the first two subsets in $\bar\pi$ are $\{a-1,a\},\{a+1\}$
 and in $\overline{\pi \sigma_1}$: $\{a,a+1\}, \{a-1\}$. Letting
$\pi=[a,a-1,a+1,\dots]$ one gets $\pi \sigma_1 \pi^{-1}
=(a-1,a+1)$, so the separating hyperplane is $x_{a-1}=x_{a+1}$.

For $1<i<n-3$, an edge is oriented from $f^{-1}(\bar\pi)$ to
$f^{-1}(\overline{\pi \sigma_i})$ if and only if the chord labeled
$i-1$ is $[a-k,a+m]$ and the latter triangulation is obtained from
the first by flipping a diagonal $[a-k-1,a+m]$; namely, by
replacing the diagonal $[a-k-1,a+m]$ by $[a-k,a+m+1]$. Then the
$i$-th and $i+1$-st subsets in $\bar\pi$, which are $\{a-k-1\},
\{a+m\}$, are switched in $\overline{\pi \sigma_{i+1}}$. So
$\pi(i)=a-k-1$ and $\pi(i+1)=a+m+1$, and
 $\pi \sigma_{i+1}
\pi^{-1} =(a-k-1,a+m)$.

Finally, an edge is oriented from $f^{-1}(\bar\pi)$ to
$f^{-1}(\overline{\pi \sigma_{n-3}})$ if and only if the chord
labeled $n-5$ is $[b-2,b+1]$ and the latter triangulation is
obtained from the first by flipping the short chord labeled $n-4$
$[b-1,b+1]$; namely, by replacing the diagonal $[b-1,b+1]$ by
$[b,b+2]$. Then the last two subsets in $\bar\pi$ are
$\{b-1\},\{b,b+1\}$ and in $\overline{\pi \sigma_{n-3}}$:
$\{b\},\{b-1,b\}$. So $\pi=[\dots,b-1,b+1,b]$ one gets $\pi
\sigma_{n-3} \pi^{-1} =[b-1,b+1]$.

\end{proof}


\section{Proof of Corollary~\ref{RS}}

Recall that $T$ and $T^R$ are antipodes (Theorem~\ref{antipodes}).


\begin{proposition}\label{reverse-2}
For every colored triangle-free triangulation $T\in CTFT(n)$, the
corresponding chambers in $\AAA(K'_n)$ satisfy
$$
c_{f(T^R)}=-c_{f(T)}.
$$
\end{proposition}

\begin{proof}
Let $w_0:=[n,n-1,n-2,\dots.1]$ be the longest permutation in
$S_n$. It follows from Observation~\ref{reverse} that
for every $\pi\in U_n$
$$
(f^{-1}(\bar \pi))^R=f^{-1}(\overline{\pi w_0}).
$$
Notice that the points $\pi$ and $\pi w_0$ in $\bbr^n$ belong to
negative chambers $c$ and $-c$. The proof is completed.

\end{proof}

\medskip

\noindent{\bf Proof of Corollary~\ref{RS}.}
By Proposition~\ref{reverse-2}, the set of hyperplanes which
separate the chamber $c_{f(T)}$ from the chamber $c_{f(T^R)}$ is
the set of all hyperplanes in $\AAA(K'_n)$. By Theorem~\ref{main}
together with Lemma~\ref{coherent-orientations}, one deduces that
all diagonals have to be flipped at least once in a geodesic from
$T$ to $T^R$. Finally, by Theorems~\ref{diameter}
and~\ref{antipodes}, the distance between $T$ and $T^R$ in
$\Gamma_n$ is equal to the number of diagonals in a convex
$n$-gon. Hence, each diagonal is flipped exactly once.

\qed

\section{Geodesics and Shifted Tableaux}\label{5}

Let $T_0$ be the {\em canonical colored star triangle-free
triangulation}; that is the triangulation, which consists of the
chords $[0,2],[0,3],\dots,[0,n-2]$ labeled $0,\dots,n-4$
respectively.

\subsection{Order on the Diagonals}\ \\


By Corollary~\ref{RS}, every geodesic from the canonical colored
star traingle-free triangulation $T_0$ to $T_0^R$ determines a
linear order on the diagonals. The following theorem characterizes
these linear orders.

\medskip



\begin{theorem}\label{main2} 
An order on the set of diagonals
$\{[i,j]:\ 1\le i < j-1 \le n-1\}$
of a convex $n$-gon
appears in geodesics in
$\Gamma_n$ from 
$T_0$ to its reverse $T_0^R$ if and only if it is a linear
extension of the coordinate-wise order with respect to the natural
order
$$
0<1<2<\cdots<n-1,
$$
or its reverse
$$
0\equiv n<n-1<n-2<\cdots<1.
$$

\end{theorem}


\begin{proof}
Clearly, every geodesic from $T_0$ to $T_0^R$ starts with either
flipping $[0,2]$ or $[0,n-2]$. By symmetry, exactly half start by
flipping $[0,2]$.

\smallskip

First, we will prove that an order on the set of diagonals of a
convex $n$-gon corresponding to geodesics from $T_0$ to $T_0^R$,
which start by flipping $[0,2]$, is a linear extension of the
coordinate-wise order with respect to the natural order
$ 0<1<2<\cdots<n-1$. 

\smallskip

Recall that by Corollary~\ref{RS}, every hyperplane is not crossed
more than once. Thus, in order to prove this, it suffices to show
that in every gallery from the identity chamber
$c_{[0,1,2,\dots,n-1]}$ to its negative, that start by crossing
the hyperplane $H_{0,2}$, the hyperplane $H_{k,l}$ is crossed
after the hyperplane $H_{i,j}$, whenever $i+1<j$, $k+1<l$, and
$(i,j)<(k,l)$ in point-wise coordinate order. In other words, it
suffices to prove that for every arc permutation $\pi\in U_n$, if
$\bar\pi\in U'_{n}$ corresponds to a chamber in such a gallery
then $\pi^{-1}(i)<\pi^{-1}(j) \implies \pi^{-1}(k)<\pi^{-1}(l)$.

Clearly, this holds for the arc permutations which correspond to
the identity chamber $c_{[0,1,\dots,n-1]}$ and to its negative
$-c_{[0,1,\dots,n-1]}=c_{[n-1,n-2,\dots,0]}$. With regard to all
other chambers in these galleries,
notice first, that all geodesics from $T_0$ to $T_0^R$ must end by
flipping either $[n-3,n-1]$ or $[1,3]$. Let $S\in CTFT(n)$ be the
triangulation, which consists of the chords $[1,3],
[0,3],[0,4],\dots,[0,n-2]$ labeled $0,1,\dots,n-4$ respectively.
Then $T_0^R$ is obtained from $S^R$ by flipping $[1,3]$ and $S$ is
obtained from $T_0$ by flipping $[0,2]$. Since $S$ and $S^R$ are
antipodes, it follows that $S$ does not appear in a geodesic from
$T_0$ to $T_0^R$ which start by flipping $[0,2]$. Thus every
geodesic from $T_0$ to $T_0^R$, which start by flipping $[0,2]$
must end by flipping $[n-3,n-1]$.
One concludes that
for every $\pi\in U_n$, if $\bar\pi\in U'_{n}$ is a chamber in
such a gallery which is not first or last, then
$\pi^{-1}(2)<\pi^{-1}(0)$ and $\pi^{-1}(n-3)<\pi^{-1}(n-1)$.
Thus the first letter in $\pi$, $\pi(0)$, is not $0$ or $n-1$.
There are three cases to analyze:

If $\pi(0)=1$ then, since $\pi^{-1}(n-3)<\pi^{-1}(n-1)$ and $\pi$
is an arc permutation,
$0=\pi^{-1}(1)<\pi^{-1}(2)<\cdots<\pi^{-1}(n-3)<\pi^{-1}(n-1))$.

If $\pi(0)=n-2$ then, since $\pi^{-1}(2)<\pi^{-1}(0)$ and $\pi$ is
an arc permutation,
$0=\pi^{-1}(n-2)<\pi^{-1}(n-3)<\cdots<\pi^{-1}(2)<\pi^{-1}(0)$.

Finally, if $2\le \pi (0)\le n-3$ then,  since
$\pi^{-1}(n-3)<\pi^{-1}(n-1)$, $\pi^{-1}(2)<\pi^{-1}(0)$ and $\pi$
is an arc permutation, letting $\pi(0):=i$ the following holds:
$0=\pi^{-1}(i)<\pi^{-1}(i-1)\cdots<\pi^{-1}(2)<\pi^{-1}(0)$ and
$0=\pi^{-1}(i)<\pi^{-1}(i+1)<\cdots<\pi^{-1}(n-3)<\pi^{-1}(n-1)$.

%
%

\smallskip

It follows that for every $\pi\in U_n$, such that $\bar\pi$ is a
chamber in a gallery from the identity chamber to its negative
that start by flipping $[0,2]$,
 there is no
$(i,j)<(k,l)$ in point-wise coordinate order, with $i+1<j$ and
$k+1<l$, such that $\pi^{-1}(i)<\pi^{-1}(j)$ but
$\pi^{-1}(k)>\pi^{-1}(l)$.  One concludes that there is no
$(i,j)<(k,l)$ in point-wise coordinate order, with $i+1<j$ and
$k+1<l$,
such that $[i,j]$ is flipped after $[k,l]$.

\medskip

It remains to prove the opposite direction, namely, to show that
every linear extension of the coordinate-wise order appears as a
geodesic. To prove this, first, notice that the lexicographic
order does appear. Then observe that if $i<j<k<l$ and $[i,l]$ and
$[j,k]$ are consequent flipped diagonals in the geodesic then it
is possible to switch their order in the geodesic. This completes
the proof for geodesics from $T_0$ to $T_0^R$, which start by
flipping $[0,2]$.

\medskip

Finally, to prove that geodesics from $T_0$ to $T_0^R$, which
start by flipping $[0,n-2]$, are characterized by linear
extensions with respect to the order $ 0\equiv
n<n-1<n-2<\cdots<1$, observe that these geodesics  may be obtained
from geodesics that start by flipping $[0,2]$ via the reflection
which maps every $0\le i \le n-1$ to $n-i$.


\end{proof}






\subsection{Skew Shifted Young Lattice}\label{KHYT}

\begin{defn}
For a positive integer $n$ let $\Lambda(n)$ be the set of all
partitions with largest part $\le n$ and with all parts distinct,
except possibly the first two parts when they are equal to $n$.
Namely,
\begin{eqnarray*}
\Lambda(n):= \{ \lambda=(\lambda_1,\dots,\lambda_k)&:& k\ge0,\
n\ge \lambda_1\ge \lambda_2>\lambda_3>\cdots>\lambda_k>0 \hbox{ and }\\
& & [\hbox{either } \lambda_1>\lambda_2 \hbox{ or } \lambda_1=
\lambda_2 =n]\}.
\end{eqnarray*}
Let $(\Lambda(n), \subseteq)$ the poset of partitions in
$\Lambda(n)$ ordered by inclusion of the corresponding Young
diagrams.
\end{defn}

\begin{example}
\begin{eqnarray*}
\Lambda(3)= &\{&(3,3,2,1),\ (3,3,2),\ (3,3,1),\ (3,3),\\
& &(3,2,1),\ (3,2),\ (3,1),\ (3),\ (2,1),\ (2),\ (1), \ ()\quad\}.
\end{eqnarray*}
\end{example}



\bigskip

Consider the standard tableaux of
 truncated shifted staircase shape $(n-1,n-1,n-2,n-3,\dots,2,1)$.
Denote this set by $Y(n)$.

\medskip

\begin{example}\label{ex5} 
\rm
The truncated shifted staircase shape
$(3,3,2,1)$ is drawn in the following way:
$$
\begin{array}{ccccc}
X         & X         & X           & *              \\
          & X         & X           & X    \\
          &           & X           & X     \\
          &           &             & X      \\
\end{array}
$$

\smallskip

\noindent There are four standard tableaux of this shape
\[
\begin{array}{ccccc}
1         & 2         & 3           & *              \\
          & 4         & 5           & 6    \\
          &           & 7           & 8     \\
          &           &             & 9      \\
\end{array}
\;,\;
\begin{array}{ccccc}
1         & 2         & 4           & *  \\
          & 3         & 5           & 6  \\
          &           & 7           & 8  \\
          &           &             & 9    \\
\end{array}
\;,\;
\begin{array}{ccccc}
1         & 2         & 3           & *              \\
          & 4         & 5           & 7    \\
          &           & 6           & 8     \\
          &           &             & 9      \\
\end{array}
\;,\;
\begin{array}{ccccc}
1         & 2         & 4           & *              \\
          & 3         & 5           & 7    \\
          &           & 6           & 8     \\
          &           &             & 9      \\
\end{array}
\]
\end{example}

\bigskip

\begin{observation}\label{observation-lambda}
\begin{itemize}
\item[1.] The maximal chains in $(\Lambda(n),\subseteq)$ are
parameterized by the set of standard tableaux of
 truncated shifted staircase shape $(n-1,n-1,n-2,\dots,1)$.
\item[2.] The linear extensions of the coordinate-wise order on
the set
$$
\{(i,j):\ 0\le i+1< j\le n\}\setminus \{(0,n)\}
$$
are parameterized by the set of standard tableaux of
 truncated shifted staircase shape $(n-1,n-1,n-2,\dots,1)$.
\end{itemize}
\end{observation}

With any standard tableau of truncated shifted staircase shape $T$
associate two words of size ${n\choose 2}-1$, $r(T)$ and $c(T)$,
where $r(T)_i$ ($c(T)_i$), $(1\le i\le {n\choose 2}-1)$,
is the row (respectively, column) where entry $i$ is located.

\begin{example}\label{ex6}\rm
Let $P,Q$ be the first two tableaux in Example~\ref{ex5}. Then
$r(P)=(1,1,1,2,2,2,3,3,4)$, $c(P)=(1,2,3,2,3,4,3,4,4)$,
$r(Q)=(1,1,2,1,2,2,3,3,4)$ and $c(Q)=(1,2,2,3,3,4,3,4,4)$,
\end{example}

\bigskip

\subsection{Geodesics and Tableaux}\ \\

%

Denote the set of geodesics from $T_0\in CTFT(n)$ to $T_0^R$
starting by flipping $[0,2]$ by $D(T_0)^+$.

\begin{proposition}\label{SYT-G}
\begin{itemize}
\item[1.] There is a bijection from the set of geodesics
$D(T_0)^+$ to $Y(n-3)$ (the set of standard tableaux on truncated
shifted staircase partition $(n-3,n-3,n-4,\dots,1)$)
$$
\phi:D(T_0)^+\to Y(n-3).
$$
\item[2.] For every geodesic $u\in D(T_0)^+$, the diagonal flipped
at the $i$-th step is
$$
[r(\phi(u))_i-1, c(\phi(u))_i+1].
$$
\end{itemize}
\end{proposition}

\begin{example}
The bijection $\phi$ maps the tableau
\[
\begin{array}{ccccc}
1         & 2         & 4           & *  \\
          & 3         & 5           & 6  \\
          &           & 7           & 8  \\
          &           &             & 9    \\
\end{array}
\]
to the series of diagonals: $[0,2], [0,3], [1,3], [0,4], [1,4],
[1,5], [2,4], [2,5], [3,5].$
\end{example}

\begin{proof} Combining Theorem~\ref{main2} with Observation~\ref{observation-lambda}(2).
\end{proof}






\bigskip

Let $d_n$ denote the number of geodesics from the canonical star
triangulation $T_0$ of an $n$-gon to its reverse $T_0^R$.
By Proposition~\ref{SYT-G}, $d_n/2$ is equal to the number of
standard tableaux of truncated shifted staircase shape
$(n-3,n-3,,n-4,\dots,1)$.
%
%
Partial results regarding $d_n$ were stated in an early version of
this preprint. Subsequently, an explicit multiplicative formula
was proved by Greta Panova~\cite{Panova} and Ronald C. King and
the authors~\cite{AKR}.

\begin{theorem}
The number of geodesics from the canonical star triangulation
$T_0$ of a convex $n$-gon to its reverse $T_0^R$ is
$$
d_n= g^{[n-6]} \cdot {N \choose 4n-15}\cdot \frac{8(2n-9)}{n-3} =
\frac{N! \cdot 8(2n-9)}{(4n-15)! \cdot (n-3)} \cdot
\prod_{i=0}^{n-7} \frac{i!}{(2i+1)!},
$$
where  $g^{[n-6]}:=g^{(n-6,n-7,\dots,1)}$ is the number of
standard Young tableaux of shifted staircase shape
$(n-6,n-7,\dots,1)$ and $N := n(n-3)/2$.
\end{theorem}





\end{document}